\documentclass[12pt]{article}

\usepackage{cmap}  
\usepackage[T2A]{fontenc}
\usepackage[utf8]{inputenc}
\usepackage[usenames]{color}
\usepackage{amsmath,amssymb,amsthm}
\usepackage{arcs}
\usepackage{epsfig}
\usepackage{graphicx}
\usepackage[pdftex,colorlinks=true,linkcolor=blue,urlcolor=red,unicode=true,hyperfootnotes=false,bookmarksnumbered]{hyperref}
\usepackage{ifthen}
\usepackage{import}
\usepackage{indentfirst}
\usepackage{mathtools}
\usepackage{cancel}
\usepackage{xcolor}

\newcommand{\eps}{\varepsilon}

\renewcommand{\le}{\leqslant}
\renewcommand{\ge}{\geqslant}

\newcommand{\E}{\ensuremath{\mathsf{E}}}

\newcommand{\Var}{\ensuremath{\mathsf{Var}}}
\newcommand{\Prb}{\ensuremath{\mathsf{P}}}
\newcommand{\ff}{\mathcal F}
\newcommand{\GG}{\mathcal G}

\newtheorem{thm}{Theorem}
\newtheorem{prob}[thm]{Problem}

\newtheorem{cor}[thm]{Corollary}
\newtheorem{prop}[thm]{Proposition}
\newtheorem{conj}{Conjecture}

\begin{document}
\title{Rainbow matchings in $k$-partite hypergraphs}
\author{Sergei Kiselev\footnote{Laboratory of Combinatorial and Geometric Structures, Moscow Institute of Physics and Technology, Email: {\tt kiselev.sg@gmail.com}}, Andrey Kupavskii\footnote{G-SCOP, CNRS, Grenoble; Laboratory of Combinatorial and Geometric Structures, Moscow Institute of Physics and Technology; IAS Princeton, Email: {\tt kupavskii@ya.ru}}}

\maketitle
\begin{abstract}
  In this paper, we prove a conjecture of Aharoni and Howard on the existence of rainbow (transversal) matchings in sufficiently large families  $\ff_1,\ldots, \ff_s$ of tuples in $\{1,\ldots, n\}^k$, provided $s\ge 470.$
\end{abstract}
\section{Introduction}

Let $[n]:=\{1,\ldots,n\}$. In this paper, we study families of $k$-tuples $\ff\subset [n]^k$. We say that $F = (f_1,\ldots, f_k)$ and $F' = (f'_1,\ldots, f'_k)$, $F,F' \in [n]^k,$ \emph{intersect} iff for some $1 \le i \le k$ we have $f_i = f_i'$. Note that $[n]^k$ can be seen as the complete $k$-partite $k$-uniform hypergraph with parts of size $n$. The following conjecture was made by Aharoni and Howard \cite{AH}:
\begin{conj}\label{conj1} Let $n, s$ and $k$ be positive integers. If $\ff_1,\ldots,\ff_s\subset [n]^k$ satisfy $|\ff_i|>(s-1)n^{k-1}$ for all $i$ then there exist $F_1\in\ff_1,\ldots, F_s\in \ff_s,$ such that $F_i\cap F_j = \emptyset$ for any $1\le i<j\le s.$
\end{conj}
We call any such collection of $F_i$ a {\it rainbow $s$-matching.} If true, the bound on $|\ff_i|$ in the conjecture is best possible: consider the families $\ff_1 = \ldots = \ff_s = \{F=(f_1,\ldots, f_k)\in [n]^k: f_1\in [s-1]\}$. In their paper, Aharoni and Howard proved this conjecture for $k=2,3$. Later, Lu and Yu \cite{LY} proved it for $n>3(s-1)(k-1).$

The main result of this paper is the proof of Conjecture~\ref{conj1} for all $s\ge s_0.$

\begin{thm}\label{thm1}
There exists $s_0$ such that Conjecture~\ref{conj1} is true for any $s\ge s_0$.
\end{thm}
Although we haven't put much effort into optimizing $s_0$, our proof allows us to take  $s_0 = 470$.

The proof of Theorem~\ref{thm1} relies on the idea that intersection of any family with a random matching is highly concentrated around its expectation. This idea was introduced in the paper of Frankl and the second author \cite{FK3} in the context of the Erd\H os Matching Conjecture (Erd\H os, \cite{EMC}). Let ${[n]\choose k}$ stand for the collection of all $k$-element subsets of $[n]$. The EMC states that the largest family $\ff\subset {[n]\choose k}$ with no $s$ pairwise disjoint sets for $n\ge sk$ has size at most $f(n,k,s):=\max\big\{{n\choose k}-{n-s+1\choose k}, {sk-1\choose k}\big\}$. 
The conjecture is proven in some ranges of parameters. For somewhat large $n$, the best results on the conjecture is due to Frankl \cite{F4}, who showed the validity of the conjecture for roughly $n>2sk$ and Frankl and the second author \cite{FK3}, who showed the conjecture is valid for $n>\frac 53 sk$ and $s>s_0$. For $n$ close to $sk$, the only nontrivial result is due to Frankl \cite{F5}, who showed the validity of the conjecture for $n<s(k+\epsilon),$ where $\epsilon=\epsilon(k)$ is roughly $k^{-k}$. A more general problem was studied in \cite{FK5}.

Interestingly, the 'one-family' version of the conjecture of Aharoni and Howard is almost trivial. Indeed, let us show that any family $\ff\subset [n]^k$ of size greater than $(s-1)n^{k-1}$ has a matching of size $s$. Consider a perfect matching in $[n]^k$ taken uniformly at random. Obviously, it consists of $n$ edges. Thus, via simple averaging, the expected intersection of $\ff$ with this matching is  $n\cdot \frac {|\ff|}{n^k}>s-1.$ Therefore, there exists a perfect matching that has at least $s$ sets in the intersection with $\ff$. This should explain the intuition behind using the concentration for the intersections of families with a random matching. Indeed, if for each of the families $\ff_1,\ldots,\ff_s$ their intersection with a perfect matching is either of size $s-1$ or $s$, then essentially the same argument as above will imply that we can find a rainbow $s$-matching in a fixed perfect matching, as long as one of the families intersects it in $s$ sets.

Huang, Loh and Sudakov \cite{HLS}, as well as Aharoni and Howard \cite{AH} suggested the $s$-family analogue of the EMC: any $s$ families $\ff_1,\ldots, \ff_s\subset {[n]\choose k}$ with $\min_{i}|\ff_i|>f(n,k,s)$ contain a rainbow $s$-matching. (We refer to it as to {\it rainbow }EMC later on.) Huang, Loh and Sudakov proved it in the very same range: for $n>3sk^2$. (Later, their approach was transposed in the aforementioned paper of Lu and Yu \cite{LY} to progress on Conjecture~\ref{conj1}. Note here that the bound $n>3sk$ in the case of $[n]^k$ corresponds to the bound $n>3sk^2$ in the case of ${[n]\choose k}$ since the ground set has size $nk$ in the former case.) The approaches of \cite{F4} and \cite{FK3}, unfortunately, do not seem to work for the rainbow version of the EMC. Using a junta method, Keller and Lifshitz \cite{KL} showed the validity of the rainbow EMC for $n>f(s)k,$ where $f(s)$ is an unspecified and a very quickly growing function of $s$. Together with Frankl \cite{FK4}, the second author managed to prove a much better junta approximation for shifted families and to show the validity of the rainbow EMC for $n>12sk(2+\log s).$ We also note that the validity of the conjecture for $n>Csk$ with some large unspecified $C$ is announced by Keevash, Lifshitz, Long and Minzer.

In the next section, we give the necessary preliminaries and prove the aforementioned concentration inequality. In Section~\ref{sec3}, we give the proof of Theorem~\ref{thm1}.


\section{Concentration for intersections with a random matching}
We start with some results concerning the eigenvalues of graphs and their relation to their quasirandomness properties, as well as concentration results for martingales.


Consider the graph $PG_{n,k}$ on the vertex set $[n]^k$ and with the edge set consisting of all pairs of disjoint sets. It is not difficult to see that $PG_{n,k}$ is the direct product of $k$ copies of $K_n$, the complete graph on $n$ vertices. Thus, the adjacency matrix of $PG_{n,k}$ is the $k$-th Kronecker power of $(J_n-E_n)$, the adjacency matrix of $K_n$. For two $n\times n$ matrices $A$ and $B$ with eigenvalues $\lambda_i$ and $\mu_j$, respectively, the eigenvalues of their Kronecker product $A\otimes B$ are $\lambda_i\mu_j$.

This implies the following proposition.
\begin{prop}
The largest eigenvalue of $PG_{n,k}$ is equal to $D := (n-1)^{k}$ and the second largest absolute value of an eigenvalue is equal to $\lambda_2 := (n-1)^{k-1}$.
\end{prop}
\begin{proof} This readily follows from the paragraph above and the fact that the adjacency matrix of $K_n$ has the form $J_n-E_n$ and its eigenvalues are $(n-1)$, with multiplicity $1$, and $-1$, with multiplicity $(n-1)$.\end{proof}

Fix integers $n, k$. Let $\mathcal G \subset [n]^k$ be a family and set $\alpha := |\mathcal G| / n^k$. Take two disjoint sets $S_1,S_2\in [n]^k$ uniformly at random and let $A_i$ be the event that $S_i\in \GG$. 
Now we will need the following result of Alon and Chung \cite{AC}:
\begin{thm}
\label{alon_chung}
Let $G = (V, E)$ be a $D$-regular graph on $m$ vertices, let $\lambda_2(G)$ be the second largest absolute value of the eigenvalues of $G$ and let $S$ be a subset of vertices $G$ of cardinality $|S| = \alpha m$. Then the following holds:
$$
\left|\frac{2e(S)}{Dm} - \alpha^2 \right| \le 
\frac{\lambda_2(G) \alpha (1 - \alpha)}{D},
$$
where $e(S)$ is the number of edges which join two vertices of $S$.
\end{thm}

 We apply Theorem~\ref{alon_chung} to the graph $PG_{n,k}$. Note that, in our case, $\Prb[A_1\cap A_2]$ is exactly the proportion of all edges of $PG_{n,k}$ that are contained in $\mathcal G$. We also use that $\Prb[\bar A_1\cap A_2]= \Prb[A_2]-\Prb[A_1\cap A_2] = \alpha-\Prb[A_1\cap A_2]$. Putting this together, we get the following.

\begin{prop} \label{prob}
    $|\Prb[A_1\cap A_2] - \alpha^2|, |\Prb[\bar A_1\cap A_2] - \alpha(1-\alpha)|\le \frac{\lambda_2 \alpha(1-\alpha)}{D} =  \frac{\alpha(1-\alpha)}{n-1}$.
\end{prop}

Next, we state a result on martingales that will be used. In their survey, Chung and Lu proved the following concentration inequality (cf. \cite[Theorems 7.1 and 7.2]{CL}):

\begin{thm}
\label{m_conc_1}
Suppose that a nonnegative martingale $X$, associated with a filter $\mathbf{F}$, satisfies
$$
\Var (X_i \mid \ff_{i-1}) \le  \phi_i X_{i-1}
$$
and
$$
|X_i - X_{i-1}| \le M
$$
for $1 \le i \le n$. Here, $\phi_i$ and $M$ are nonnegative constants. Then for $\delta\in\{-1,1\}$ we have
    $$
    \Prb\big[\delta\cdot (X_n - \E X) \ge \lambda\big] \le
    \exp\left( -\frac{\lambda^2}{ 2((\E X + \lambda) (\sum_{i=1}^n \phi_i) + M\lambda/3) } \right).
    $$



\end{thm}


\subsection{The concentration}

Fix integers $n, k$. Let $\mathcal G \subset [n]^k$ be a family and set $\alpha = |\mathcal G| / n^k$. Let $\eta$ be the random variable $|\mathcal G \cap \mathcal M|$, where $\mathcal M = (M_1, \ldots, M_n)$ is chosen uniformly at random out of all ordered perfect matchings. Clearly, we have $\E \eta = \alpha n$.

The goal of this section is to prove the following theorem.
\begin{thm}\label{thmconcen} For any $\lambda > 0$ and $\delta\in\{-1, 1\}$ the following holds.
\begin{equation}\label{concen1}
\Prb[\delta\cdot(\eta - \alpha n) \ge 2\lambda] \le
2 \exp\left(- \frac{\lambda^2}{\alpha n / 2 + 2\lambda} \right).
\end{equation}
Moreover, for any $\mu>0$,  $\lambda\ge 2\sqrt{\alpha n} + 1\ge 9$ and $\delta\in \{-1,1\}$, we have
\begin{equation}\label{concen2}
\Prb\big[\delta\cdot(\eta-\alpha n)>2\lambda+2\mu]\le
\Prb\big[\delta\cdot(\eta-\alpha n) \ge 1 \big]\cdot
    4\exp\Big(-\frac{\mu^2}{(\alpha n+2\lambda)/2+2\mu}\Big).
\end{equation}

\end{thm}
While the first part of the statement is a large deviations statement for $\eta$, the second part of the statement is a 'conditional large deviations' result, which tells us that `very large deviations should be much less probable than moderately large deviations'. This part of the statement is technical, but required in the proof of the main theorem.

Put $\eta := \eta_1 + \ldots + \eta_n$ and $\tilde\eta := \eta_1 + \ldots + \eta_{n/2}$, where $\eta_i$ is the indicator function of the event $A_i$ that the set $M_i$ belongs to $\mathcal G$.
(In what follows, we assume that $n$ is divisible by $2$. This simplifying assumption only slightly affects the calculations to follow and does not affect the result.)

Although ultimately we are interested in the behaviour of $\eta$, for technical reasons we need to deal with $\tilde \eta.$ To this end, we may relate them as follows.

\begin{prop}\label{prophalf} Let $a>0$ be a real number, $\delta\in \{-1,1\}$. Then $\Prb[\delta\cdot(\eta-\E\eta)\ge 2a]\le 2\Prb[\delta\cdot(\tilde \eta-\E\tilde\eta)\ge a]$.
\end{prop}
\begin{proof}
The proof of the inequality for both choises of $\delta$ follows from the easy fact that $\eta = \tilde\eta+\tilde\eta',$ where  $\tilde\eta'$ is a copy of $\tilde\eta$. (Note that $\tilde \eta$ and $\tilde \eta'$ are, in general, dependent.)
\end{proof}

\bigskip

Let $X_0, \ldots, X_{n/2}$ and $Y_0, \ldots, Y_{n/2}$ be the following exposure martingales:
$$
X_i = \E [\tilde\eta \mid \eta_1, \ldots, \eta_i] \ \ \ \text{and}\ \ \ Y_i = \E [\eta_{n} \mid \eta_1, \ldots, \eta_i].$$
Note that $X_i = \sum_{j=1}^i \eta_j + (n/2 - i) Y_i$.

\begin{prop}\label{propforconc}
The following holds for any $i\le n-2$:

    \begin{enumerate}
        \item $\Var [\E[\eta_{n} \mid \eta_{i+1} ,M_1,\ldots,M_{i}] \mid M_1, \ldots, M_{i}] \le \frac{\E[\eta_{n}\mid M_1,\ldots,M_{i}]}{(n-i-1)^2}$.
        \item $\big|\E[\eta_{n}\mid \eta_{i+1} ,M_1,\ldots,M_{i}] - \E[\eta_{n}\mid M_1,\ldots,M_{i}]\big| \le \frac{1}{n - i - 1}$.
    \end{enumerate}
\end{prop}
Note that in both 1. and 2. we are comparing some functions pointwise, i.e., the inequalities should hold for any choices of $M_1,\ldots, M_i$ and $\eta_{i+1}$.

\begin{proof}[Proof of Proposition~\ref{propforconc}]
Fix $M_1, \ldots, M_{i}$ and consider the graph $PG_{n', k}$ on $V' := [n]^k \setminus \bigcup_{j=1}^{i} M_j \cong [n']^k$, where $n' := n - i$. Put $\mathcal G' := \mathcal G \cap V'$ and $\alpha' := |\mathcal G'| / (n')^k$. Consider the intersection of $\mathcal G'$ with a randomly chosen $n'$-matching $\mathcal M' = (M_{i+1},\ldots, M_{n})$ of sets from $[n']^k$. Let $\eta_j$ and $A_j$ be defined as before.

Note that $\E\eta_{n} = \E\eta_{i+1} = \Prb[A_{i+1}]= \alpha'.$
In these terms, we need to show \begin{equation}\label{eqnew2}\Var [\E[\eta_{n} \mid \eta_{i+1}]] \le \frac{\alpha'}{(n'-1)^2}\ \ \ \ \text{and}\ \ \ \  \big|\E [\eta_{n} \mid \eta_{i+1}]-\alpha'\big|\le \frac{1}{n'-1}\ \ \ \ \text{for } i\le n-2.\end{equation}

\bigskip

Note that $\E [\eta_n|\eta_{i+1}]$ and $\Var[\eta_n|\eta_{i+1}]$ are both random variables taking two values depending on the value of $\eta_{i+1}$. Let us first consider the case $\eta_{i+1} = 1$.
$$
\E [\eta_{n} \mid \eta_{i+1} = 1] =
\E [\eta_{n} \mid A_{i+1}] =
\Prb[A_{n} \mid A_{i+1}] =
\frac{\Prb[A_{n} \cap A_{i+1}]}{\Prb[A_{i+1}]}=\frac{\Prb[A_{n} \cap A_{i+1}]}{\alpha'}.
$$
Then, using Proposition~\ref{prob}, we conclude that
\begin{equation}\label{dev1}
\big|\E [\eta_{n} \mid \eta_{i+1} = 1] - \alpha'\big| \le
\frac{1-\alpha'}{n'-1}.
\end{equation}
Next, let us consider the case $\eta_{i+1} = 0$.
$$
\E [\eta_{n} \mid \eta_{i+1} = 0] =
\frac{\Prb[A_{n}\cap\bar A_{i+1}]}{\Prb[\bar A_{i+1}]} =
\frac{\Prb[A_{n}\cap\bar A_{i+1}]}{1 - \alpha'}.
$$
Similarly, using Proposition~\ref{prob}, we conclude that
\begin{equation}\label{dev2}
|\E [\eta_n' \mid \eta_{i+1}' = 0] - \alpha'| \le
\frac{\alpha'}{n'-1}.
\end{equation}
From \eqref{dev1} and \eqref{dev2} we conclude that the second part of \eqref{eqnew2} and thus the proposition holds.
Since $\Prb[\eta_{i+1} = 1] = \alpha'$ and $\Prb[\eta_{i+1} = 0] = 1 - \alpha'$, combining \eqref{dev1} and \eqref{dev2}, we conclude that {\small $$\Var [\eta_{n} \mid \eta_{i+1}] = \alpha'\cdot (\E[\eta_n\mid \eta_{i+1}=1]-\alpha')^2+(1-\alpha')\cdot (\E[\eta_n\mid \eta_{i+1}=0]-\alpha')^2 \le \frac{\alpha'(1 - \alpha')}{(n'-1)^2} \le \frac{\alpha'}{(n'-1)^2}.$$}
\end{proof}

Now let us note that it is easy to deduce the following corollary:
 \begin{cor}\label{corforconc}  The following holds  for any $i\le n-2$: \begin{enumerate}
        \item $\Var [Y_{i+1}\mid \eta_1, \ldots, \eta_{i}] \le \dfrac{Y_i}{(n-i-1)^2}$.
        \item $|Y_{i+1} - Y_i| \le \frac{1}{n - i - 1}$.
    \end{enumerate}
\end{cor}

\begin{proof} Informally, this corollary follows from Proposition~\ref{propforconc} via averaging over all choices of $M_1,\ldots, M_i$ that give the corresponding values of $\eta_1,\ldots, \eta_i$ and using the convexity of variance. Let us give a formal proof. Again, we need to show that a certain inequality between functions holds pointwise. In what follows, we fix a choice of $\eta_1,\ldots, \eta_i.$ It is easy to see that 
\begin{align*}
Y_{i+1} =&\ t^{-1}\sum\E[\eta_{n}\mid \eta_{i+1},M_1 = M'_1,\ldots, M_i = M'_{i}],\\
Y_{i} =&\ t^{-1}\sum\E[\eta_{n}\mid M_1 = M'_1,\ldots, M_i = M'_{i}],\end{align*} where both sums are over all tuples $(M'_1,\ldots,M'_{i})$ of disjoint $k$-sets such that $M'_i\in \mathcal G$ if and only if $\eta_j=1$, 
and $t$ is the number of such tuples. This immediately implies the second part of the corollary. As for the first part, we derive it below. We apply Jensen's inequality to the variance to get the first inequality below and use the first part of Proposition~\ref{propforconc} to get the second inequality:
\begin{align*}\Var [Y_{i+1}\mid \eta_1, \ldots, \eta_{i}]  =&\ \Var \big[t^{-1}\sum\E[\eta_{n}\mid \eta_{i+1},M_1 = M'_1,\ldots, M_1 = M'_{i}]\mid \eta_1, \ldots, \eta_{i}\big]\\ \le&\ t^{-1}\sum\Var \big[\E[\eta_{n}\mid \eta_{i+1},M_1 = M'_1,\ldots, M_i = M'_{i}]\mid \eta_1, \ldots, \eta_{i}\big]\\ \le&\  t^{-1}\sum\frac{\E[\eta_{n}\mid M_1 = M_1',\ldots,M_i = M_{i}']}{(n-i-1)^2}\\ =&\ \dfrac{Y_i}{(n-i-1)^2}.\end{align*}
\end{proof}

\begin{cor}\label{corcor}
 For any $i\le n/2-1$, we have
    $$ |X_{i+1}-X_i|\le \frac 32\ \ \ \text{and} \ \ \
    \Var [X_{i+1} \mid \eta_1, \ldots, \eta_i]\le \frac{X_i}{n}.
    $$
\end{cor}
\begin{proof}
We have  $0\le Y_i,\eta_{i+1}\le 1$ and thus $-(n/2-i-1)|Y_{i+1}-Y_i|-Y_i\le |X_{i+1}-X_i|\le \eta_{i+1}+(n/2-i-1)|Y_{i+1}-Y_i|$, which, using part 2 of Corollary~\ref{corforconc}, implies $|X_{i+1}-X_i|\le \frac 32$. As for the second inequality, using Corollary~\ref{corforconc}, we can expand it as follows.
\begin{small}
 $$
    \Var [X_{i+1} \mid \eta_1, \ldots, \eta_i]= (n/2 - i)^2 \Var [Y_{i+1} \mid \eta_1, \ldots, \eta_i] \le
    \frac{(n/2 - i)^2}{(n - i - 1)^2} Y_i \le
    \frac{n/2 - i}{(n-i-1)^2} X_i \le \frac{X_i}{n}.
    $$
We note that getting the last inequality is essentially the only reason why we had to deal with $\tilde\eta$ instead of $\eta$ itself.\end{small}
\end{proof}

Now let us apply Theorem~\ref{m_conc_1} to $X_{n/2}=\tilde\eta$. If we choose $\phi_i = \frac{1}{n}$, $M = \frac 32$, then we have
$$
\Prb[\delta\cdot (\tilde\eta - \alpha n/2) \ge \lambda] \le
\exp\left(- \frac{\lambda^2}{\alpha n / 2 + 2\lambda} \right).
$$
Combining this with Proposition~\ref{prophalf}, we get \eqref{concen1}.





\bigskip

Let us deduce \eqref{concen2} from \eqref{concen1}. Assume that $\delta = 1$ and that, for some choice of $(M_1,\ldots, M_{n}),$ we have $\eta-\E \eta>2\mu+2\lambda$. Let $Z_1,\ldots, Z_n$ be the martingale associated with $\eta,$ where $$Z_i:=\E[\eta\mid \eta_1\ldots,\eta_i].$$
Then, for some $i$, we have $2\lambda -2\le Z_i - \alpha n \le 2\lambda$. The latter follows since, as it is not difficult to see, $|Z_{i+1}-Z_i|\le 2$. To see that, one just needs to repeat the proof of Corollary~\ref{corcor} with the obvious changes that we get when passing from $\tilde \eta$ to $\eta$. 

Take any collection of disjoint edges $(M_1,\ldots, M_i)$ such that $i$ is the first such step. Let us restrict to the subgraph as in the proof of Proposition~\ref{propforconc}. Note that, within this restriction, $Z_i$ is a constant from the interval $[2\lambda-2+\alpha n,2\lambda+\alpha n]$. We apply the first part of Theorem~\ref{thmconcen} with $Z_i - \sum_{j=1}^i \eta_i = \E Z_n - \sum_{j=1}^i \eta_i \le \alpha n + 2\lambda$ playing the role of $\alpha n$. We get that
$$
\Prb\big[Z_n - \alpha n \le 0 \big]
\le
\Prb[Z_n - Z_i \le -2\lambda + 2]
\le
2\exp\Big(-\frac{(\lambda-1)^2}{(\alpha n+2\lambda)/2+2(\lambda-1)}\Big).
$$
Recall that $\lambda = 2\sqrt{\alpha n} + 1$ and that $\sqrt {\alpha n}\ge 4$. Then $(\alpha n+ 2\lambda)/2+2\lambda - 2 \le \frac83 \alpha n$ and the expression in the right hand side above is at most $2\exp(-\frac {4\alpha n}{8\alpha n/3}) = 2 e^{-3/2}<\frac 12.$ The same holds for any bigger $\lambda.$ Recall that $Z_n$ has only integer values. Therefore, in these assumptions,
$$
\Prb\big[ Z_n-\alpha n\ge 1 \big] \ge \frac 12.
$$

On the other hand, we can similarly see that

$$
\Prb\big[Z_n-\alpha n>2\lambda+2\mu]\le
\Prb[Z_n - Z_i > 2\mu]\le
2\exp\Big(-\frac{\mu^2}{(\alpha n+2\lambda)/2+2\mu}\Big).
$$

Since both displayed  formulas are valid for any choice of $i$ and $(M_1,\ldots, M_i)$ for which  the event `$\eta-\alpha n>3\lambda+2\mu$' has non-zero probability, we can combine the two displayed formulas, sum them over all possible choices of $(M_1,\ldots, M_i)$ and get that
$$
\Prb\big[\eta-\alpha n > 2\lambda+2\mu] \le
\Prb\big[\eta-\alpha n \ge 1 \big]\cdot 4\exp\Big(-\frac{\mu^2}{(\alpha n +2\lambda)/2 +2\mu}\Big)
$$
for any $\lambda\ge 2\sqrt{\alpha n} + 1\ge 9$. The case $\delta = -1$ is absolutely analogous.

\section{Proof of Theorem~\ref{thm1}}\label{sec3}

We argue indirectly.
Assume that there exist $s$ families $\ff_1,\ldots, \ff_s$ in $[n]^k$, each of size exactly $(s-1) n^{k-1}+1$, and with no rainbow $s$-matching.  Consider the following random variables:
$$
\zeta_i(\mathcal M):=|\ff_i\cap \mathcal M|-s+1,
$$
where $\mathcal M$ is a uniformly random perfect matching of $k$-element sets.

For a fixed perfect matching $\mathcal M = \{M_1,\ldots, M_n\}$, consider the graph $G(\mathcal M)$ with parts $A(\mathcal M)  = \{M_1,\ldots, M_n\}$ and $B:=\{\ff_1,\ldots, \ff_s\}$, where a family is connected to a set from the matching iff the set belongs to the family. The condition that there is no rainbow $s$-matching in $\ff_1,\ldots, \ff_s$ translates into the condition that $G(\mathcal M)$ does not have a matching of the part $B$. Thus Hall's necessary and sufficient condition for the existence of a matching of $B$ must be violated. An easy consequence of that is that there exists some $j_{\mathcal M}\in \{0,\ldots, s-1\}$, such that $\zeta_i(\mathcal M)\le -j_{\mathcal M}$ for at least $s-j_{\mathcal M}$ indices $i\in [s]$. 
To make the definition of $j_{\mathcal M}$ unambiguous, choose $j_{\mathcal M}$ to be the largest out of all possible values. We treat $j_{\mathcal M}$ as a random variable depending on $\mathcal M.$

We have 
\begin{equation}\label{eqneg}\sum_{i=1}^s\zeta_i(\mathcal M)\cdot I[\zeta_i(\mathcal M)\le 0]\le -j_{\mathcal M}(s-j_{\mathcal M})\end{equation} due to the violation of Hall's condition.

First, let us show that \begin{equation}\label{eqmain0}\E[\zeta_i\mid \zeta_i> 0]< 3.7 \sqrt{s\log s}.\end{equation} For that, we shall employ \eqref{concen2} with $\eta:=\zeta_i+ s-1$ and $\alpha n = \frac{|\ff_i|}{n^k}$. Put $\delta = 1$, $\lambda = 2\sqrt s$ and let $\mu\ge 0$. Then we have
$$
\Prb[\zeta_i> 2\mu + 4\sqrt s]\le
\Prb[\zeta_i>0] \cdot 4\exp\Big(-\frac{\mu^2}{(s+4\sqrt s)/2+2\mu}\Big).
$$
Now we rewrite $\E[\zeta_i\mid \zeta_i>0]$ as follows for some $X > 0$:
$$
\E[\zeta_i\mid \zeta_i>0] = 
\E[\zeta_i\cdot I[\zeta < X]\mid \zeta_i>0] + \E[\zeta_i\cdot I[\zeta \ge X]\mid \zeta_i>0] \le 
X + \sum_{i\ge X} \Prb[\zeta_i \ge i \mid \zeta_i>0].
$$
Then, putting $X = \sqrt{s\log s} + 4\sqrt{s}$, we get
\begin{align*}
\E[\zeta_i\mid \zeta_i>0]\ \le&\
2\sqrt{s \log s} + 4\sqrt s + \sum_{t\ge \sqrt{s\log s}:\  2t+4\sqrt s\in \mathbb N}(2t + 4\sqrt s)\cdot 4\exp\Big(-\frac{t^2}{(s+4\sqrt s)/2+2t}\Big)\\ 
\le&\
2\sqrt{s \log s} + 4\sqrt s + 4 \le
3.7 \sqrt{s \log s}.
\end{align*}
(In the last two inequalities we use that $s \ge 470$ and calculated the corresponding expressions on computer.)

\medskip


Recall that $|\ff_i| > (s-1)n^{k-1}$, and thus $\E \zeta_i = \E [|\ff_i\cap \mathcal M|]-s+1 > 0$.
Then we have 
$$
0<\sum_{i=1}^s\E\zeta_i = \sum_{i=1}^s\E[\zeta_i\mid \zeta_i\le 0]\cdot \Prb[\zeta_i\le 0]+\sum_{i=1}^s\E[\zeta_i\mid \zeta_i> 0]\cdot \Prb[\zeta_i>0].
$$
For simplicity, let us put $x := \lfloor 3.7 \sqrt{s\log s} \rfloor$. Combining this with \eqref{eqneg} and \eqref{eqmain0},  we get that
\begin{equation}\label{eqmain}
    0<-\E [j_{\mathcal M}(s-j_{\mathcal M})]+\sum_{i=1}^s x\Prb[\zeta_i>0]\le -\E [j_{\mathcal M}(s-j_{\mathcal M})]+x \E[j_{\mathcal M}]=-\E [j_{\mathcal M}(s-x-j_{\mathcal M})].
\end{equation}
At the same time, using that $j_{\mathcal M} < s,$ we get 
\begin{equation}\label{eqmain_2}
\E\big[j_{\mathcal M}\cdot I[j_{\mathcal M}<s-x]\big]=\E j_{\mathcal M}-\E\big[j_{\mathcal M}\cdot I[j_{\mathcal M}\ge s-x]\big]\ge \E j_{\mathcal M}-s \Prb[j_{\mathcal M}\ge s-x].
\end{equation}
From here, we get the following chain of inequalities (note that we use $j_{\mathcal M}(s-x-j_{\mathcal M})\ge -sx$ for $j_{\mathcal M}\le s$): 
\begin{align*}\E [j_{\mathcal M}(s-x-j_{\mathcal M})]=&\ \E\big[j_{\mathcal M}(s-x-j_{\mathcal M})\cdot I[j_{\mathcal M}<s-x]\big]+\E\big[j_{\mathcal M}(s-x-j_{\mathcal M})\cdot I[j_{\mathcal M}\ge s-x]\big]\\ \ge&\ \E\big[j_{\mathcal M}\cdot I[j_{\mathcal M}<s-x]\big] -sx\Prb[j_{\mathcal M}\ge s-x]\\ \overset{\eqref{eqmain_2}}{\ge}&\ \E j_{\mathcal M}-s(x+1)\Prb[j_{\mathcal M}\ge s-x]. 
\end{align*}
Together with \eqref{eqmain}, this gives
\begin{equation}\label{eqmain2}s(x+1)\Prb[j_{\mathcal M}\ge s-x]> \E j_{\mathcal M}.\end{equation} 


Note that, by definition, $\Prb[j_{\mathcal M}\ge s-x]\le \sum_{i=1}^s\Prb[\zeta_i\le x - s]$. Combining this with \eqref{eqmain0} and \eqref{eqmain2}, we get
\begin{equation}
    \label{eqmain3}
    \sum_{i=1}^s\E\big[\zeta_i\cdot I[\zeta_i>0]\big]\overset{\eqref{eqmain0}}{\le}
    x\E j_{\mathcal M} \overset{\eqref{eqmain2}}{<}
    sx(x+1)\Prb[j_{\mathcal M}\ge s-x]\le sx(x+1) \sum_{i=1}^s \Prb[\zeta_i\le x-s].
\end{equation}

We shall apply \eqref{concen2} to $\eta = \zeta_i+(s-1)$. This time, we have $\delta=-1$,  $\mu=\frac{s - x - 4\sqrt s}{2}$, $\lambda = 2\sqrt s$ we get that
\begin{align}
\Prb[\zeta_i\le x-s]=&\ 
\Prb[\eta-(s-1) \le x-s] \le
\Prb[\eta\le s-2]\cdot 4\exp\Big(-\frac{(s - x - 4\sqrt{s})^2 }{6 s - 4x - 8\sqrt{s}}\Big)
\notag\\
\label{eqmain4}
\le&\ 
\Prb[\zeta_i \le -1]\cdot 4e^{-z(s)} \le
-4e^{-z(s)}\E[\zeta_i\cdot I[\zeta_i\le 0]],
\end{align}
where $z(s) = \frac{(s - x - 4\sqrt{s})^2 }{6 s - 4x - 8\sqrt{s}} \sim \frac{s}{6}$ (up to terms of order $\sqrt{s\log s}$).
Combining \eqref{eqmain3} with \eqref{eqmain4} and substituting the value of $x$, we get that
$$
\sum_{i=1}^s\E\big[\zeta_i\cdot I[\zeta_i>0]\big]\le
-sx(x+1) \cdot 4 e^{-z(s)} \sum_{i=1}^s \E\big[\zeta_i\cdot I[\zeta_i\le 0]\big].$$
Once $sx(x+1) \cdot 4 e^{-z(s)}<1$, this contradicts the assumption $\sum_{i=1}^s \E\zeta_i>0$. This happens for any $s>470$. It is easy to see that all the other inequalities used in the proof are valid for this choice of $s.$ The theorem is proved.

\section{Discussion}

As we have mentioned in the introduction, the example $\ff_1 = \ldots = \ff_s = [s-1] \times [n]^{k-1}$ shows that the uniform lower bound in Conjecture~\ref{conj1} is tight. In what follows, we explore what happens if we bound the size of each family separately.

\begin{prob}
Let $n \ge s$ and $k$ be positive integers. We call a sequence $f_1 \le \ldots \le f_s$ \emph{satisfying} if any  $\ff_1, \ldots, \ff_s \subset [n]^k$ with $|\ff_i| > f_i$ contain a rainbow matching.
Which sequences $f_1,\ldots,f_s$ are satisfying?
\end{prob}

Theorem~\ref{thmconcen} easily implies that, for some $C > 0$, the sequence $a_1,\ldots, a_s$ with $a_i := (i + C \sqrt{s\log s})n^{k-1}$ is satisfying. (In fact, it implies that with very large probability, for all $i$ simultaneously the family $\ff_i$ intersects a random matching in at least $i$ sets.) A modification of the proof of Theorem~\ref{thm1} implies that the sequence $b_1,\ldots, b_s$ with $b_i := \min(i + C \sqrt{s\log s}, s-1+\eps)n^{k-1}$ is satisfying, where $\eps>0$ can be chosen arbitrarily small, and we need to either take $C \ge C_0(\eps)$ or $s\ge s_0(\eps)$.

It is tempting to suggest that the sequence $0,n^{k-1},\ldots, (s-1)n^{k-1}$ is satisfying. However, this is not the case. Indeed, even the sequence $c_1,\ldots, c_{s}$ with $c_1 =\ldots = c_{s-1} = (s-1)n^{k-1} - (n-1)^{k-1} - 1$, $c_s = (s-1) n^{k-1}$, is not. Indeed, fix a set $F\in [s+1, n]\times [n]^{k-1}$ and consider the following families: $\ff_1 = \ldots = \ff_{s-1} = [s-1] \times [n]^{k-1} \setminus \{F'\in
\{1\}\times[n]^{k-1}\colon F\cap F' = \varnothing\}$, $\ff_s = [s-1]\times [n]^{k-1} \cup \{F\}$.


\begin{conj}
The sequences $d_1,\ldots, d_s$ with $d_i := \min(i + C \sqrt{s\log s}, s-1)n^{k-1}$ and $e_1,\ldots, e_s$ with  $e_i:= i\cdot n^{k-1}$ are satisfying.
\end{conj}

Note that the sequence $d_1,\ldots, d_s$ being satisfying implies Conjecture~\ref{conj1}.

Finally, we note that we may ask a variant of Conjecture~\ref{conj1} for $\ff_1,\ldots, \ff_s\subset [n_1]\times [n_2]\times\ldots\times [n_k]$ for $n_1\le\ldots\le n_k$. There, a natural extremal example would be to fix $s-1$ coordinates in the first part, which is the smallest, and consider the family of all tuples that intersect the first part in one of those coordinates.

A simple induction argument allows to reduce this, seemingly more general, version of Conjecture~\ref{conj1}, to the case $n_1 = \ldots = n_k$. Let us give a sketch of that. Take  families $\ff_1,\ldots, \ff_s\subset [n_1]\times \ldots\times [n_k]$ that do not contain a rainbow $s$-matching. First, using a standard argument, we may assume that the families are {\it shifted} in each part, i.e., if $(f_1,\ldots,f_k)\in \ff_i$ then $(f_1',\ldots, f'_k)\in \ff_i$ for any $f'_1,\ldots, f'_k$ satisfying $f'_j\le f_j$ for every $j$. Given this, we may proceed by induction on $k$ for fixed $s$ and on $\sum_i n_i$ for fixed $s$ and $k$, as long as $n_1<n_k$. We apply the induction hypothesis to $\ff_i(\bar n_k):=\{(a_1,\ldots,a_{k})\in \ff_i: a_k\ne n_k\} $. 
It is easy to see that $\ff_1(\bar n_k),\ldots, \ff_s(\bar n_k)$ do not contain a rainbow $s$-matching, and thus $\min_i|\ff_i|\le (s-1)\cdot n_2\cdot\ldots\cdot n_{k-1}\cdot (n_k-1)$. At the same time, for every $i$, shiftedness imply that the degree of the element $\{n_k\}$ in the $k$-th part is the smallest among $\{1,\ldots, n_k\}$, and thus $|\ff_i(\bar n_k)|\ge \frac {n_k-1}{n_k}|\ff_i|$. Combining these two inequalities, we get that $\min_i|\ff_i|\le (s-1)\cdot n_2\cdot\ldots\cdot n_k,$ as required.

\vskip+0.4cm
{\sc Acknowledgements: }
The authors acknowledge the financial support from the Ministry of Education and Science of the Russian Federation in the framework of MegaGrant no 075-15-2019-1926.

 The research of the second author was directly supported by the IAS Fund for Math, the Director’s Fund and indirectly supported by the National Science Foundation Grant No. CCF-1900460. Any opinions, findings and conclusions or recommendations expressed in this material are those of the authors and do not necessarily reflect the views of the National Science Foundation.

\end{document}